\documentclass{amsart}
\usepackage{tikz}
\usepackage{hyperref}   
\usepackage{amssymb,xcolor}
\usepackage{stmaryrd} 

\usepackage{mathtools}  
\mathtoolsset{showonlyrefs}

\theoremstyle{plain}
\newtheorem{corollary}{Corollary}
\theoremstyle{plain}
\newtheorem{theorem}{Theorem}[section]
\theoremstyle{plain}
\newtheorem{definition}[theorem]{Definition}
\theoremstyle{plain}
\newtheorem{lemma}[theorem]{Lemma}
\theoremstyle{remark}
\newtheorem{remark}[theorem]{Remark}
\theoremstyle{plain}
\newtheorem{proposition}[theorem]{Proposition}
\numberwithin{equation}{section}
\theoremstyle{plain}

\begin{document}

\title[Curvature estimate of stable free boundary MOTS]{A curvature estimate
for stable marginally outer trapped hypersurface with a free boundary}

\author{Xiaoxiang Chai}
\address{Korea Institute for Advanced Study, Seoul 02455, South Korea}
\email{xxchai@kias.re.kr}

\begin{abstract}
  A marginally outer trapped hypersurface is a generalization of minimal
  hypersurfaces originated from general relativity. We show a curvature
  estimate for stable marginally outer trapped hypersurfaces up to the free
  boundary satisfying a uniform area bound. Our proof is based on an iteration
  argument. The curvature estimate was previously known via a blowup argument
  for stable minimal hypersurfaces.
\end{abstract}



{\maketitle}

\section{Introduction}

Let $M^n$ be a spacelike submanifold in $\mathcal{S}^{n + 1, 1}$ and $l^{\pm}$
be the two independent future directed null sections of the normal bundle of
$M$ with corresponding null second fundamental form or shear tensor
$\chi^{\pm}$. The traces of $\chi^{\pm}$ are called the null expansions which
we denote them by $\theta^{\pm}$.

\begin{definition}
  The submanifold $M$ is called marginally outer (inner) trapped if
  \begin{equation}
    \theta^{\pm} = 0. \label{mots}
  \end{equation}
  We call $M$ a MOTS (MITS) in short.
\end{definition}

We will only consider the case when $M$ sits in a spacelike hypersurface $N^{n
+ 1} \subset \mathcal{S}^{n + 1, 1}$. Let $\tau$ be the future timelike normal
of $N$, $\hat{g}$ be the induced metric on $N$ and $p$ the second fundamental
form of $N$ with respect to $\tau$ in $\mathcal{S}^{n + 1, 1}$. The triple
$(N^{n + 1}, \hat{g}, p)$ is usually refered as \text{{\itshape{an initial
data set}}}. A rather natural choice of $l^{\pm}$ in this situation is $\nu
\pm \tau$ where $\nu$ is the outward pointing normal of $M$ in $N$. Let $e_i$
be an orthonormal frame of the tangent space of $M$, then
\begin{equation}
  \chi_{\pm} = \sum_{i = 1}^n \langle \nabla_{e_i} (\nu \pm \tau), e_i
  \rangle,
\end{equation}
an we have the null expansion is given by
\begin{equation}
  \theta^{\pm} = H \pm \ensuremath{\operatorname{tr}}_M p, \label{mots in
  spacelike hypersurface}
\end{equation}
where $H$ is the mean curvature of $M$ in $N$ and
$\ensuremath{\operatorname{tr}}_M p$ is the trace of the projection of $p$ to
$M$.

The MOTS equation \eqref{mots} is then a prescribed mean curvature equation.
When $N$ is a time-symmetric Cauchy hypersurface in $\mathcal{S}^{n + 1, 1}$
i.e. $p \equiv 0$, then the a MOTS is simply minimal in $N$. The concept of
stability is central in the theory of MOTS which extends the stability of
stable minimal hypersurfaces.

\begin{definition}[{\cite{andersson-local-2005}}]
  We say that a MOTS $M$ is stable if there is an outward infinitesimal
  deformation which does not decrease $\theta^+$. To write the notion
  analytically, there exists a non-zero function $f \geqslant 0$ such that
  \begin{equation}
    \delta_{f \nu} \theta^+ \label{mots regrouped stability} \geqslant 0.
  \end{equation}
\end{definition}

For the construction of marginally outer trapped surface, see
{\cite{eichmair-plateau-2009,andersson-area-2009,an-emergence-2020,an-anisotropic-2021,roesch-mean-2022}}.
The notion of a free boundary (more generally, capillary) MOTS as well as the
stability was introduced by Alaee-Lesourd-Yau {\cite{alaee-stable-2020}}.
Recall the definition of a free boundary stable marginally outer trapped
hypersurface and its stability.

\begin{definition}[{\cite{alaee-stable-2020}}]
  A hypersurface $M$ with $\partial M \neq \emptyset$ is said to be a free
  boundary marginally outer trapped hypersurface if $\theta^+ = 0$ and
  $\partial M$ meets $\partial N$ orthogonally. It is said to be stable if
  there exists a non-zero $f \geqslant 0$ such that
  \begin{equation}
    \delta_{f \nu} \theta^+ \geqslant 0 \text{ in } M, \delta_{f \nu} \langle
    \eta, \nu \rangle = 0 \text{ on } \partial M, \label{stability of free
    boundary mots}
  \end{equation}
  where $\eta$ be the outward normal of $\partial N$ in $N$.
\end{definition}

The equation \eqref{stability of free boundary mots} can be written down in
terms of inequalities for $f$ as done in Lemma \ref{first variation of null
expansion}. The angle forming by the two vector fields $\eta$ and $\nu$ is
call the contact angle of $\partial M$ and $\partial N$. Geometrically, the
infinitesimal deformation in \eqref{stability of free boundary mots} fixes the
contact angle of $\partial M$ and $\partial N$.

To simplify, we assume that $N$ lies in a larger manifold $\tilde{N}$ of the
same dimension with boundary such that any point $x \in \partial N$ there
exists a geodesic ball $B (x, r) \subset \tilde{N}$. To define the geodesic
balls centered at the boundary point is easier than the balls defined via the
Fermi coordinates (see {\cite[Appendix A]{li-min-max-2021}}).

Now we state our main result which is the following curvature estimate for
stable marginally outer trapped hypersurface assuming a volume bound.

\begin{theorem}
  \label{main curvature estimate}Let $2 \leqslant n \leqslant 5$, if $M^n$ is
  a stable free boundary marginally outer trapped hypersurface in $(N,
  \hat{g}, p)$ satisfying a uniform volume bound
  \begin{equation}
    \ensuremath{\operatorname{vol}} (B (x, r) \cap M) \leqslant C_M r^n,
    \label{volume growth}
  \end{equation}
  for all $x \in \bar{M}$ and all $0 < r \leqslant r_0$, then there exists a
  bound on the curvature
  \begin{equation}
    |A| (x) \leqslant \tfrac{C}{r} \label{curvature estimate}
  \end{equation}
  with $C$ depending only on $C_M$, $|p|_{C^1}$,
  $|\ensuremath{\operatorname{Rm}}|_{C^0}$, $\ensuremath{\operatorname{inj}}
  (\tilde{N}, g)^{- 1}$ and $|d|_{C^3}$. Here $\ensuremath{\operatorname{Rm}}$
  is the curvature operator of $\tilde{N}$, $d$ is the distance function to
  $\partial N$ in $\tilde{N}$ and $\ensuremath{\operatorname{inj}} (\tilde{N},
  g)$ is the injective radius of $\tilde{N}$.
\end{theorem}

\begin{remark}
  The curvature estimate holds as well if we replace the extrinsic balls with
  intrinsic balls in \eqref{volume growth}. The dependence on $|d|_{C^3}$ is
  actually dependence on $|b|_{C^1}$ where $b$ is the second fundamental form
  of the boundary $\partial N$ in $N$. We write in this way because $b$ is
  simply the Hessian of $d$.
\end{remark}

Guang-Li-Zhou {\cite{guang-curvature-2020}} showed a curvature estimate for
stable minimal hypersurfaces using blow-up argument. Our approach involves
deriving a Simons inequality for a perturbation of the second fundamental form
(see Definition \ref{perturbation}), combining with the stability
\eqref{stability of free boundary mots} for the stable free boundary MOTS. It
is in spirit closer to the works Schoen-Simon-Yau
{\cite{schoen-curvature-1975}} and Andersson-Metzger
{\cite{andersson-curvature-2010}}. Compared to the work
{\cite{guang-curvature-2020}}, the constant in our curvature estimate
\eqref{curvature estimate} explicitly depends on the geometric quantities of
$N$, $\partial N$ and the volume bound. The technical difference between our
work and {\cite{andersson-curvature-2010}} is that we do not consider shear
tensor $\chi^{\pm}$. The reason lies in the presence of the free boundary.

Schoen-Simon {\cite{schoen-regularity-1981}} generalized the curvature
estimate {\cite{schoen-curvature-1975}} for embeded stable minimal
hypersurfaces to any dimension. However, Schoen-Simon theory for high
dimensional stable free boundary minimal hypersurfaces is not seen in
literature yet.

We would also like to mention another important problem: the curvature
estimate for immersed stable minimal or marginally outer trapped surfaces with
a free boundary (see {\cite[Conjecture 1.4]{guang-curvature-2020}}) without
area bound. This is the free boundary analog of Schoen's curvature estimate
{\cite{schoen-estimates-1983,colding-estimates-2002}}. It is also quite
interesting to seek a curvature estimate for immersed stable capillary
surfaces in an arbitrary manifold with boundary.

The paper is organized as follows:

In Section \ref{prelim}, we collect basics on Simons identity, the
perturbation of the second fundamental form and its boundary derivatives. In
Section \ref{stability section}, we derive some integral estimates only by the
stability condition \eqref{stability of free boundary mots}. In Section
\ref{simons section}, we calculate a Simons inequality for the perturbation of
the second fundamental form. In Section \ref{curvature estimate section}, via
a de Giorgi iteration, we conclude the proof for the pointwise curvature
estimate. In the Appendix \ref{app:Sobolev}, we record the Sobolev inequality
on free boundary hypersurfaces.

\

\text{{\bfseries{Acknowledgments}}} I would like to thank Sven Hirsch (Duke),
Martin Lesourd (Harvard) and Martin Li (CUHK) for discussions. My research is
supported by the KIAS research grant under the code MG074402.

\section{Preliminaries}\label{prelim}

First, we fix some notations used in this article. Let $K$ be the Riemann
curvature operator of $N$, $R$ be that of $M$, $\hat{\nabla}$ be the
Levi-Civita connection of $(N, \hat{g})$, $D$ be the induced connection on
$\partial N$, and $\nabla$ the induced connection on $M$.

We collect a few facts which would be used frequently later in the work.

\subsection{Simons identity}

We recall the Simons' identity.

\begin{theorem}
  \label{simons}For any hypersurface $M$ in $N$ the second fundamental form
  $h_{i j}$ satisfies the identity
\begin{align}
& \Delta h_{i j} \\
= & \nabla_i \nabla_j H - \nabla_j K_{n + 1, k i k} - \nabla_k K_{n + 1, i
j k} \\
& - K_{k j i l} h_{k l} - K_{k j k l} h_{i l} - |A|^2 h_{i j} + h_{i l}
h_{j l} H. \label{ssy 1.20}
\end{align}
\end{theorem}

\begin{proof}
  The identity is due to {\cite{simons-minimal-1968}}, see also
  {\cite[(1.19)-(1.20)]{schoen-curvature-1975}}. In these papers the identity
  is not in the form that we will need, so we give a quick derivation for the
  convenience of the reader. We pick an orthonormal frame $\{e_i \}_{1
  \leqslant i \leqslant n}$ on $T M$. We use the Einstein summation convention
  where the summation is done on repeated indices.
  
  First, by Codazzi equation,
  \begin{equation}
    \Delta h_{i j} = \nabla_k \nabla_k h_{i j} = \nabla_k \nabla_j h_{i k} +
    \nabla_k K_{j k i, n + 1},
  \end{equation}
  By commuting covariant derivatives and the Gauss equation,
\begin{align}
& \nabla_k \nabla_j h_{i k} \\
= & \nabla_j \nabla_k h_{i k} - R_{k j i l} h_{k l} - R_{k j k l} h_{i l}
\\
= & \nabla_j \nabla_k h_{i k} - (K_{k j i l} + h_{j i} h_{k l} - h_{j l}
h_{i k}) h_{k l} - (K_{k j k l} + h_{j k} h_{k l} - h_{j l} h_{k k}) h_{i
l} \\
= & \nabla_j \nabla_k h_{i k} - (K_{k j i l} + h_{j i} h_{k l} - h_{j l}
h_{i k}) h_{k l} - (K_{k j k l} + h_{j k} h_{k l} - h_{j l} h_{k k}) h_{i
l} \\
= & \nabla_j \nabla_k h_{i k} - K_{k j i l} h_{k l} - K_{k j k l} h_{i l}
- |A|^2 h_{i j} + h_{i l} h_{j l} H.
\end{align}
  Applying the Codazzi equation on $\nabla_j \nabla_k h_{i k}$ again,
  \begin{equation}
    \nabla_j \nabla_k h_{i k} = \nabla_j \nabla_i h_{k k} - \nabla_j K_{i k k,
    n + 1} = \nabla_j \nabla_i H + \nabla_j K_{i k k, n + 1} .
  \end{equation}
  Collecting all the above we obtained the desired identity for $\Delta h_{i
  j}$.
\end{proof}

\subsection{Perturbed second fundamental form}

Next, we collect some facts about the perturbation we are going to use
throughout the paper.

\begin{definition}[{\cite{edelen-convexity-2016}}]
  \label{perturbation}Extend and fix $k$ and $\eta$ to be defined on all of
  $M$. Define the perturbed second fundamental form $\bar{A}$ of $M$ to be
  \begin{equation}
    \bar{h}_{i j} = h_{i j} + \bar{T}_{i j} : = h_{i j} + T_{i j \nu} +
    \Lambda_0 g_{i j}
  \end{equation}
  where $T$ is a 3-tensor define on $N$ by
  \begin{equation}
    T (X, Y, Z) = b (X, Z) g (Y, \eta) + b (Y, Z) g (X, \eta),
  \end{equation}
  and $D_0$ is a constant depending only on $|b|_{C^0}$ such that
  \[ T (X, X, \nu) + \Lambda_0 \geqslant 1 + |p|_{C^0} \label{D0 size} \]
  for any unit vector $X$. The perturbed null second fundamental form or shear
  tensor is given by
  \begin{equation}
    \bar{\chi}_{i j} = \bar{h}_{i j} + p_{i j} . \label{perturbed null second
    fundamental form}
  \end{equation}
\end{definition}

The most significant feature of the perturbed second fundamental form is that
$\bar{h} (\eta, \cdot)$ vanishes when restricted to $\partial M$ and that the
perturbed second fundamental form $\bar{h}$ is comparable with the original
second fundamental form and with good control of the boundary derivatives.
These are presented in the following two lemmas and Proposition \ref{pp h k
relation}.

\begin{lemma}
  \label{lemma:derivatives of T}Let $M$ be a hypersurface in $N$ not
  neccessarily a MOTS, $T$ be a 3-tensor on $N$, and $T_{i j \nu}$ be the
  2-tensor $T (\cdot, \cdot, \nu)$ restricted to $T M$, then
  \begin{equation}
    | \nabla T_{i j \nu} | \leqslant c_1 (1 + |A|) . \label{derivatives of T}
  \end{equation}
  The constant $c_1 > 0$ depends only on $|T|_{C^1}$. And for any 2-tensor $p$
  on $N$,
  \begin{equation}
    | \nabla p| \leqslant c_2 (1 + |A|), \label{derivatives of p}
  \end{equation}
  The constant $c_2 > 0$ depends on $|p|_{C^1}$.
\end{lemma}

\begin{proof}
  The computations are tensorial, we may assume the simplifications that there
  is an orthonormal frame such that $\nabla_i e_j = 0$ at some point $x \in
  M$, and hence $\bar{\nabla}_i e_j = - h_{i j} \nu$ and $\bar{\nabla}_i \nu =
  h_{i j} e_j$ at $x$. With these simplications, we have
\begin{align}
& (\nabla_k T) (e_i, e_j, \nu) \\
= & \nabla_k (T (e_i, e_j, \nu)) \\
= & \bar{\nabla}_k (T (e_i, e_j, \nu)) \\
= & (\bar{\nabla}_k T) (e_i, e_j, \nu) + T (\bar{\nabla}_k e_i, e_j, \nu)
+ T (e_i, \bar{\nabla}_k e_j, \nu) + T (e_i, e_j, \bar{\nabla}_k \nu)
\\
= & (\bar{\nabla}_k T) (e_i, e_j, \nu) - h_{i k} T_{\nu j \nu} - h_{j k}
T_{i \nu \nu} + h_{k \ell} T_{i j \ell} .
\end{align}
  So \eqref{derivatives of T} follows with $c_1$ depending on $|T|_{C^1}$.
  Similar calculation gives
  \[ (\nabla_k p) (e_i, e_j) = (\bar{\nabla}_k p) (e_i, e_j) - h_{i k} p_{j
     \nu} - h_{j k} p_{i \nu} . \]
  The bound on $| \nabla p|$ then easily follows.
\end{proof}

\begin{remark}
  See also {\cite[Proposition 5.1]{edelen-convexity-2016}} for the computation
  of $\nabla^2 T_{i j \nu}$ in $\mathbb{R}^n$.
\end{remark}

\begin{lemma}
  For a marginally outer trapped hypersurface $M$,
  \begin{equation}
    | \bar{A} | \geqslant 1, |A| \leqslant c_1 + | \bar{A} |, | \nabla A|
    \leqslant c_2 (| \nabla \bar{A} | + | \bar{A} |) . \label{perturbed and
    original relation}
  \end{equation}
  Here the constant $c_1 > 0$ depends only on $|b|_{C^0}$ and$| \eta |_{C^0}$,
  and $c_2$ depends on $|b|_{C^1}$ and $| \eta |_{C^1}$.
\end{lemma}

\begin{proof}
  First by \eqref{D0 size},
  \begin{equation}
    \bar{H} \geqslant H + n + n |p|_{C^0} \geqslant n.
  \end{equation}
  So $| \bar{A} | \geqslant \tfrac{| \bar{H} |}{n} \geqslant 1$,
  \begin{equation}
    |A| \leqslant | \bar{A} | + | \bar{T} | \leqslant | \bar{A} | + c_1
  \end{equation}
  and
  \begin{equation}
    | \nabla A| \leqslant | \nabla \bar{A} | + | \nabla \bar{T} | \leqslant |
    \nabla \bar{A} | + c (1 + |A|) \leqslant c_2 (| \nabla \bar{A} | + |
    \bar{A} |) .
  \end{equation}
  The dependence of constants $c_1$ and $c_2$ are easy to track.
\end{proof}

\subsection{Boundary derivatives}

Let $\partial_i$ and $\partial_j$ are coordinate vector fields on $M$ such
that $\eta = \partial_1$ along $\partial M$. We calculate a few important
boundary derivatives of the second fundamental form namely $\nabla_1 h_{11}$
and $\nabla_1 h_{i j}$ where $i, j > 2$.

\begin{proposition}
  \label{pp h k relation}Let $X$ and $Y$ be tangent vector fields to $\partial
  M$, then
  \begin{equation}
    h (\eta, X) = - b (\nu, X) \label{h k relation} .
  \end{equation}
  The normal derivatives of $h$ satisfies
  \begin{equation}
    (\nabla_{\eta} h) (\eta, \eta) = \nabla_{\eta} H
    -\ensuremath{\operatorname{tr}}_{\partial M} ((\nabla_{\eta} h) (\cdot,
    \cdot)) . \label{boundary derivative of h 11}
  \end{equation}
  Let $D$ be the induced connection on $\partial N$, then
\begin{align}
& (\nabla_{\eta} h) (X, Y) \\
= & - K (Y, \eta, X, \nu) - K (Y, \nu, X, \eta) \label{boundary derivative
of h ij} \\
& - h (\nabla_Y \eta, X) - b (D_X \nu, Y) - (\nabla_{\nu} b) (X, Y)
\\
& + b (X, Y) h_{\eta \eta} + A (X, Y) b_{\nu \nu} .
\end{align}
\end{proposition}

\begin{proof}
  By definitions of $h$, $k$ and the free boundary condition on $M$,
  \begin{equation}
    h (\eta, X) = \langle \bar{\nabla}_X \nu, \eta \rangle = - \langle
    \bar{\nabla}_X \eta, \nu \rangle = - B (\nu, X) .
  \end{equation}
  This proves \eqref{h k relation}. Note that $H$ vanishes identically, so
  \[ \nabla_{\eta} H = (g^{i j} - \eta^i \eta^j + \eta^i \eta^j) \nabla_{\eta}
     h_{i j}, \]
  and we get \eqref{boundary derivative of h 11}. We calculate now the two
  terms $\nabla_Y (h (\eta, X))$ and $\nabla_Y (k (\nu, X))$. First,
\begin{align}
& \nabla_Y (h (\eta, X)) \\
= & (\nabla_Y h) (\eta, X) + h (\nabla_Y \eta, X) + h (\eta, \nabla_Y X)
\\
= & (\nabla_{\eta} h) (X, Y) + K (Y, \eta, X, \nu) + h (\nabla_Y \eta, X)
+ h (\eta, \nabla_Y X) .
\end{align}
  where in the last line we used the Codazzi equation. Similarly,
\begin{align}
& \nabla_Y (b (\nu, X)) = D_Y (b (\nu, X)) \\
= & (D_Y b) (\nu, X) + b (\nu, D_Y X) + b (D_Y \nu, X) \\
= & (D_{\nu} b) (X, Y) + K (Y, \nu, X, \eta) + b (\nu, D_Y X) + b (D_Y
\nu, X) .
\end{align}
  We see that
\begin{align}
\nabla_Y X = & (\bar{\nabla}_Y X)^{\partial M} + \langle \nabla_Y X, \eta
\rangle \eta \\
= & (\bar{\nabla}_Y X)^{\partial M} - b (X, Y) \eta .
\end{align}
  and similarly $D_Y X = (\bar{\nabla}_Y X)^{\partial M} - A (X, Y) \nu$. By
  the relation \eqref{h k relation}, we have
  \begin{equation}
    h (\eta, \nabla_Y X) + b (\nu, D_Y X) = - b (X, Y) h_{\eta \eta} - A (X,
    Y) b_{\nu \nu} . \label{h k relation on derivatives}
  \end{equation}
  We know from $\nabla_Y (h (\eta, X))$ and $\nabla_X (b (\nu, X))$ that
\begin{align}
& (\nabla_{\eta} h) (X, Y) \\
= & - K (Y, \eta, X, \nu) - h (\nabla_Y \eta, X) - h (\eta, \nabla_Y X)
\\
& - (D_{\nu} b) (X, Y) - K (Y, \nu, X, \eta) - b (\nu, \nabla_X Y) - b
(D_X \nu, Y) .
\end{align}
  Using \eqref{h k relation on derivatives}, we can drop the terms containing
  $\nabla_X Y$, and we obtain \eqref{boundary derivative of h ij} for
  $(\nabla_{\eta} h) (X, Y)$.
\end{proof}

\begin{remark}
  See {\cite[Lemma 6.1]{edelen-convexity-2016}} for the case in
  $\mathbb{R}^n$, {\cite{lambert-inverse-2016}} for a special case where
  $\partial N$ is the sphere, and {\cite{hirsch-contracting-2020}} for a
  5-parameter perturbation.
\end{remark}

\begin{remark}
  Let $\{e_i \}$ (with $i$ in the appropriate range) be an orthonormal frame
  of $\partial M$, we have
  \begin{equation}
    h (\nabla_Y \eta, X) = \sum_i h (e_i, X) \langle \nabla_Y \eta, e_i
    \rangle = \sum_i h (e_i, X) b (Y, e_i),
  \end{equation}
  and similarly, $b (D_Y \nu, X) = \sum_i b (e_i, X) h (e_i, Y)$.
\end{remark}

It is easy then to see the following.

\begin{corollary}
  \label{boundary derivative bound}If $M$ is a free boundary MOTS, then
  \begin{equation}
    - c | \bar{A} | \leqslant \partial_1 | \bar{A} | \leqslant c | \bar{A} |,
  \end{equation}
  where $c > 0$ depends only on $|b|_{C^1}$.
\end{corollary}

\begin{proof}
  We know that $| \bar{A} | > 0$. So
\begin{align}
& \partial_1 | \bar{A} |^2 \\
= & \sum_{i, j \geqslant 1}^n 2 \bar{h}_{i j} \nabla_1 \bar{h}_{i j}
\\
= & 2 \bar{h}_{11} \nabla_1 \bar{h}_{11} + 2 \sum_{i, j \geqslant 2}
\bar{h}_{i j} \nabla_1 \bar{h}_{i j} + 2 \sum_{i \geqslant 2} \bar{h}_{1
i} \nabla_1 \bar{h}_{1 i} \\
= & 2 \bar{h}_{11} \nabla_1 \bar{h}_{11} + 2 \sum_{i, j \geqslant 2}
\bar{h}_{i j} \nabla_1 \bar{h}_{i j} .
\end{align}
  where in the last line we have used $\bar{h}_{1 i} \equiv 0$ along $\partial
  M$. By the MOTS equation \eqref{mots in spacelike hypersurface} and
  \eqref{derivatives of p}
  \begin{equation}
    | \nabla_{\eta} H| \leqslant | \nabla H| = | \nabla
    \ensuremath{\operatorname{tr}}_M p| \leqslant c |A| \leqslant c | \bar{A}
    | .
  \end{equation}
  So from \eqref{boundary derivative of h 11} and \eqref{boundary derivative
  of h ij}, we get
  \begin{equation}
    | \nabla_1 \bar{h}_{11} | + | \nabla_1 \bar{h}_{i j} | \leqslant c (1 + |
    \bar{A} |) .
  \end{equation}
  So due to \eqref{perturbed and original relation}, we see $| \partial_1 |
  \bar{A} || \leqslant c | \bar{A} |$.
\end{proof}

The paper {\cite{andersson-curvature-2010}} estimate the size of the shear $|
\chi |$, at a first glance, it is naturual trying to estimate the perturbed
shear tensor $| \bar{\chi} |$ defined in \eqref{perturbed null second
fundamental form}. However, when computing $\partial_1 | \bar{\chi} |$, the
term $\sum_{i \geqslant 2} \bar{\chi}_{1 i} \nabla_1 \bar{\chi}_{1 i}$ is not
favorable due to non-vanishing $\bar{\chi}_{1 i}$. The boundary derivative
$\nabla_1 \bar{\chi}_{1 i}$ then essentially requires an estimate on $\nabla_1
h_{1 i}$ in terms of $|A|$ and its estimate seems difficult to do. This is the
reason that we consider directly the perturbed second fundamental form
$\bar{A}$.

\section{Stability inequality}\label{stability section}

The work of {\cite{galloway-generalization-2006}} observed that the stability
of a closed MOTS and the dominant energy condition implies some topological
properties of the MOTS. The stability of {\cite{galloway-generalization-2006}}
used the Schoen-Yau's rewrite of the stability and contains the scalar
curvature of the MOTS. Hence, the stability \eqref{mots regrouped stability}
is not well suited for curvature estimates. We will consider the less
famililar form. To this end we recall the first variation of the null
expansion $\theta^+$ (see {\cite{andersson-jangs-2010}}) and the variation of
the contact angle {\cite{stahl-convergence-1996,alaee-stable-2020}}.

\begin{lemma}
  \label{first variation of null expansion}Let $M$ be a free boundary MOTS,
  then the first variation of the null expansion $\theta^+$ is given by
  \begin{equation}
    \delta_{f \nu} \theta^+ |_M = L f := - \Delta f + 2 S (\nabla f) - |A|^2 f
    + f\ensuremath{\operatorname{div}}S - \langle h, p \rangle_M +
    f\mathcal{X}
  \end{equation}
  where $\mathcal{X}= -\ensuremath{\operatorname{Ric}} (\nu) +
  \bar{\nabla}_{\nu} (\ensuremath{\operatorname{tr}}_N p)
  -\ensuremath{\operatorname{div}}_N p (\nu) + H p_{\nu \nu}$ and $S$ is the
  1-form $k (\nu, \cdot)$ restricted to $M$. Along the boundary $\partial M$,
  the variation of $\langle \eta, \nu \rangle$ is given by
  \begin{equation}
    \delta_{f \nu} \langle \eta, \nu \rangle = B f : = - \nabla_{\eta} f + f b
    (\nu, \nu) .
  \end{equation}
\end{lemma}

\begin{remark}
  Note that $|\mathcal{X}|  \leqslant C$ with $C$ depending on
  $|\ensuremath{\operatorname{Ric}}|_{C^0}$ and $|p|_{C^1}$.
\end{remark}

We have an integral estimate for $| \bar{A} |$ following from the stability.

\begin{lemma}
  \label{l2 estimate}If $M$ is a stable MOTS with a free boundary, the for all
  $\varepsilon > 0$ and $\phi \in C_c^{\infty} (M)$, the following inequality
  \begin{equation}
    \int_M \phi^2 | \bar{A} |^2 \leqslant (1 + \varepsilon) \int_M | \nabla
    \phi |^2 + \left[ c_1 \int_M \phi^2 + c_2 \int_{\partial M} \phi^2 \right]
    \label{initial l2 estimate}
  \end{equation}
  holds where the constant $c_1$ depends on $1 / \varepsilon$,
  $|\ensuremath{\operatorname{Ric}}|_{C^0}$, $|p|_{C^1}$ and $c_2$ depends on
  $1 / \varepsilon$, $|p|_{C^1}$ and $|b|_{C^0}$.
\end{lemma}

\begin{proof}
  We reorder the terms in Lemma \ref{first variation of null expansion} as in
  {\cite{galloway-generalization-2006}}, we see
  \begin{equation}
    0 \leqslant f^{- 1} L f =\ensuremath{\operatorname{div}} \left( S -
    \tfrac{\nabla f}{f} \right) - \left| S - \tfrac{\nabla f}{f} \right|^2 +
    |S|^2 - |A|^2 - \langle h, p \rangle_M + f\mathcal{X}.
  \end{equation}
  Multiplying the above with $\phi^2$ and integration by parts,
\begin{align}
& \int_M \phi^2 [|S - f^{- 1} \nabla f|^2 + |A|^2] \\
\leqslant & \int_{\partial M} \phi^2 \left\langle S - \tfrac{\nabla f}{f},
\eta \right\rangle + \int_M [|S|^2 - \langle h, p \rangle_M +\mathcal{X}]
\phi^2 - 2 \left\langle S - \tfrac{\nabla f}{f}, \nabla \phi \right\rangle
\phi .
\end{align}
  Using Cauchy-Schwarz inequality,
  \[ 2 \left| \left\langle S - \tfrac{\nabla f}{f}, \nabla \phi \right\rangle
     \phi \right| \leqslant \left| S - \tfrac{\nabla f}{f} \right|^2 \phi^2 +
     | \nabla \phi |^2, \]
  and the boundary stability condition in \eqref{stability of free boundary
  mots},
  \[ \partial_{\eta} f = b (\nu, \nu) f, \]
  we obtain
  \[ \int_M \phi^2 |A|^2 \leqslant \int_M | \nabla \phi |^2 + [c_1 - \langle
     h, p \rangle] \phi^2 + c_2 \int_{\partial M} \phi^2, \]
  where $c_1$ depends on $|\ensuremath{\operatorname{Ric}}|_{C^0}$,
  $|p|_{C^1}$ and $c_2$ depends on $|p|_{C^1}$ and $|b|_{C^0}$. Using
  \begin{equation}
    | \langle h, p \rangle_M | \leqslant \varepsilon |A|^2 + \tfrac{1}{2
    \varepsilon} |p|^2, \label{cauchy-schwarz h p}
  \end{equation}
  we get
  \[ \int_M \phi^2 |A|^2 \leqslant (1 + \varepsilon) \int_M | \nabla \phi |^2
     + \left[ c_1 \int_M \phi^2 + c_2 \int_{\partial M} \phi^2 \right] . \]
  Since $\bar{h}_{i j} = h_{i j} + T_{i j \nu}$, we do the same as in
  \eqref{cauchy-schwarz h p} and we obtain the desired inequality for $\int
  \phi^2 | \bar{A} |^2$ (with renaming of $\varepsilon$).
\end{proof}

Now we use a common trick to get rid of the boundary terms in \eqref{initial
l2 estimate}.

\begin{corollary}
  \label{l2 estimate no boundary}If $M$ is a stable MOTS with a free boundary,
  the for all $\varepsilon > 0$ and $\phi \in C_c^{\infty} (M)$, the following
  inequality
  \begin{equation}
    \int_M \phi^2 | \bar{A} |^2 \leqslant (1 + \varepsilon) \int_M | \nabla
    \phi |^2 + c_1 \int_M \phi^2 \label{initial l2 estimate no boundary}
  \end{equation}
  holds where the constant $c_1$ depends on $1 / \varepsilon$,
  $|\ensuremath{\operatorname{Ric}}|_{C^0}$, $|p|_{C^1}$ and $c_2$ depends on
  $1 / \varepsilon$, $|p|_{C^1}$ and $|b|_{C^0}$.
\end{corollary}

\begin{proof}
  Since $\phi$ is compactly supported, $\langle D d, \eta \rangle = 1$ by the
  free boundary condition of $M$, so by divergence theorem,
\begin{align}
& \int_{\partial M} \phi^2 \\
= & \int_{\partial M} \phi^2 \langle D d, \eta \rangle \\
= & \int_M \ensuremath{\operatorname{div}}_M (\phi^2 \nabla d) \\
= & \int_M 2 \phi \langle \nabla \phi, \nabla d \rangle + \phi^2 \langle
\nabla_{e_i} \nabla d, e_i \rangle .
\end{align}
  Note that $|H| \leqslant c$ from \eqref{mots in spacelike hypersurface}, so
  by decomposition of Hessian of $M$,
  \begin{equation}
    \sum_{i = 1}^n \langle \nabla_{e_i} \nabla d, e_i \rangle = \sum_{i = 1}^n
    (D^2 d) (e_i, e_i) - H \langle D d, \nu \rangle \label{decomposition
    hessian}
  \end{equation}
  is bounded by a constant depending on $|d|_{C^2}$ and $|p|_{C^0}$. So
  \begin{equation}
    \int_{\partial M} \phi^2 \leqslant c \int_M \phi^2 + c \int_M \phi |
    \nabla \phi | .
  \end{equation}
  Using Cauchy-Schwarz inequality on the second term on the right and
  combining with \eqref{initial l2 estimate}, we have obtained \eqref{initial
  l2 estimate no boundary}.
\end{proof}

\begin{lemma}
  \label{lp estimate by gradient}Let $M$ be a stable free boundary MOTS, we
  have for any $\varepsilon > 0$ and $q \geqslant 2$, we have
\begin{align}
& \int_M \phi^2 | \bar{A} |^{q + 2} \\
\leqslant & \tfrac{q^2}{4} (1 + \varepsilon) \int_M \phi^2 | \bar{A} |^{q
- 2} | \nabla | \bar{A} ||^2 + c_1 \int_M [\phi^2 + | \nabla \phi |^2] |
\bar{A} |^q + c_2 \int_{\partial M} \phi^2 | \bar{A} |^q . \label{eq:lp
estimate by gradient}
\end{align}
  The dependence of the constants $c_i$ are the same with Lemma \ref{l2
  estimate}.
\end{lemma}

\begin{proof}
  Letting $\phi$ to be $\phi | \bar{A} |^{q / 2}$ in \eqref{initial l2
  estimate}, we have
  \begin{equation}
    \int_M \phi^2 | \bar{A} |^{q + 2} \leqslant (1 + \varepsilon) \int_M |
    \nabla (\phi | \bar{A} |^{q / 2}) |^2 + c_1 \int_M \phi^2 | \bar{A} |^q +
    c_2 \int_{\partial M} \phi^2 | \bar{A} |^q .
  \end{equation}
  We estimate $| \nabla (\phi | \bar{A} |^{q / 2}) |^2$ as follows:
\begin{align}
& | \nabla (\phi | \bar{A} |^{q / 2}) |^2 \\
= & || \bar{A} |^{q / 2} \nabla \phi + \phi \tfrac{q}{2} | \bar{A} |^{q /
2 - 1} \nabla | \bar{A} ||^2 \\
\leqslant & \tfrac{q^2}{4} (1 + \varepsilon) \phi^2 | \bar{A} |^{q - 2} |
\nabla | \bar{A} ||^2 + c (\varepsilon^{- 1}) | \nabla \phi |^2 | \bar{A}
|^q .
\end{align}
  Combing the above two inequalities and ajusting the value $\varepsilon$, we
  obtained the desired inequality.
\end{proof}

We show that integrals $\int_M \phi^2 | \bar{A} |^q$ on the boundary can be
transfered to the an integral in the interior via an application of the
divergence theorem. It works for any hypersurface $M$.

\begin{lemma}
  Let $M$ be any hypersurface, for any $q > 0$,
  \begin{equation}
    \int_{\partial M} | \bar{A} |^q \phi^2 \leqslant c \int_M \phi  | \bar{A}
    |^{q - 1} (\phi | \bar{A} | + \phi | \nabla | \bar{A} || + | \nabla \phi
    |), \label{transfer of boundary integral}
  \end{equation}
  where the constant $c$ only depends on $p$ and $|d|_{C^2}$.
\end{lemma}

\begin{proof}
  Suppose now that $\phi$ is a function compactly supported, since $\langle D
  d, \eta \rangle \equiv 1$ along $\partial M$. By the divergence theorem,
\begin{align}
& \int_{\partial M} | \bar{A} |^q \phi^2 \\
= & \int_{\partial M} | \bar{A} |^q \phi^2 \langle D d, \eta \rangle
\\
= & \int_M \ensuremath{\operatorname{div}}_M (| \bar{A} |^q \phi^2 D d)
\\
= & \int_M | \bar{A} |^q \phi^2 \ensuremath{\operatorname{div}}_M D d +
\int_M \phi^2 D d \cdot \nabla | \bar{A} |^q + 2 \int_M | \bar{A} |^q \phi
\nabla \phi \cdot D d,
\end{align}
  and then the lemma follows similarly as Corollary \ref{l2 estimate no
  boundary}.
\end{proof}

\section{Simons inequality}\label{simons section}

\subsection{Simons inequality of $\bar{A}$}

First, we have a Simons type inequality for $\bar{A}$ of a marginally outer
trapped hypersurface $M$. The Simons identity (Theorem \ref{simons}) would be
our starting point. We combine it with estimates of the perturbation. We use
$A \ast B$ to denote linear combinations of contractions of $A \otimes B$ for
convenience.

\begin{lemma}
  Let $M$ be a MOTS, then
  \begin{equation}
    \bar{h}_{i j} \Delta \bar{h}_{i j} \geqslant \bar{h} \ast (\nabla K +
    \nabla^2 H + \Delta \bar{T}) - | \bar{A} |^4 - c (| \bar{A} |^3 + | \nabla
    \bar{A} | | \bar{A} |), \label{simons for bar h}
  \end{equation}
  where the constant $c$ depends only on $|p|_{C^1}$, $|b|_{C^1}$ and
  $|K|_{C^0}$.
\end{lemma}

\begin{proof}
  First, we estimate $H$ and $\nabla H$. The mean curvature $H$ itself is
  easy, and an upper bound follows from the MOTS equation \eqref{mots},
  \begin{equation}
    |H| = | -\ensuremath{\operatorname{tr}}_M p| \leqslant n |p|_{C^0} .
    \label{mean curvature bound}
  \end{equation}
  By \eqref{derivatives of p}, we have
  \begin{equation}
    | \nabla_i H| = | - \nabla_i \ensuremath{\operatorname{tr}}_M p| = |g^{j
    k} \nabla_i p_{j k} | \leqslant c (1 + |A|) . \label{first derivative of
    H}
  \end{equation}
  Then using these estimates in the Simons identity \eqref{ssy 1.20}, we have
  that
\begin{align}
& h_{i j} \Delta h_{i j} \\
\geqslant & - h_{i j} (\nabla_j K_{n + 1, k i k} + \nabla_k K_{n + 1, i j
k}) + h_{i j} \nabla_i \nabla_j H - |A|^4 \\
& - c (1 + |A| + |A|^2 + |A|^3 + | \nabla A| |A|) .
\end{align}
  Now we consider the perturbation,
\begin{align}
& \bar{h}_{i j} \Delta \bar{h}_{i j} \\
= & (h_{i j} + \bar{T}_{i j}) (\Delta h_{i j} + \Delta \bar{T}_{i j})
\\
= & h_{i j} \Delta h_{i j} + \bar{T}_{i j} \Delta h_{i j} + (h_{i j} +
\bar{T}_{i j}) \Delta \bar{T}_{i j} \\
= & h_{i j} \Delta h_{i j} + \bar{T}_{i j} \Delta h_{i j} + \bar{h} \ast
\Delta \bar{T} .
\end{align}
  By combining the two inequalities in the above, we see that
\begin{align}
& \bar{h}_{i j} \Delta \bar{h}_{i j} \\
\geqslant & h_{i j} (- \nabla_j K_{n + 1, k i k} - \nabla_k K_{n + 1, i j
k} + \nabla_i \nabla_j H) - |A|^4 \\
& - c (1 + |A|^3 + | \nabla A| |A|) \\
& + \bar{T}_{i j} (- \nabla_j K_{n + 1, k i k} - \nabla_k K_{n + 1, i j
k} + h_{i j} \nabla_i \nabla_j H) \\
& - c (|A|^2 + |A|^3) + \bar{h} \ast \Delta \bar{T} \\
\geqslant & \bar{h} \ast (\nabla K + \nabla^2 H + \Delta \bar{T}) - |A|^4
- c (1 + |A|^3 + | \nabla A| |A|) .
\end{align}
  We use Cauchy-Schwarz inequality to absorb $|A|$ and $|A|^2$ into $|A|^3$,
  and note that $\bar{h} = h + \bar{T}$, so
  \begin{equation}
    \bar{h}_{i j} \Delta \bar{h}_{i j} \geqslant \bar{h} \ast (\nabla K +
    \nabla^2 H + \Delta \bar{T}) - |A|^4 - c (1 + |A|^3 + | \nabla A| |A|) .
  \end{equation}
  After applying \eqref{perturbed and original relation}, we obtain our
  desired inequality.
\end{proof}

The following Kato type inequality for $\bar{A}$ is standard when applying
Schoen-Simon-Yau's iterative arguments for curvature estimates. For
completeness, we include the proof. The Kato type inequality asserts a lower
bound of $| \nabla_k \bar{h}_{i j} |^2$ in terms of $| \nabla | \bar{A} ||^2$.
Our proof is taken from {\cite{schoen-curvature-1975}}.

\begin{lemma}
  \label{kato}Let $M$ be a MOTS, then we have
  \begin{equation}
    \sum_{i j k} | \nabla_k \bar{h}_{i j} |^2 - | \nabla | \bar{A} ||^2
    \geqslant \tfrac{2}{(1 + \varepsilon) n} | \nabla | \bar{A} ||^2 - c (1 +
    | \bar{A} |^2), \label{kato inequality}
  \end{equation}
  where $c > 0$ depends on $\tfrac{1}{\varepsilon}$, $|K|_{C^0}$, $|p|_{C^1}$
  and $|b|_{C^1}$.
\end{lemma}

\begin{proof}
  Let $\mathcal{T}= | \nabla \bar{A} |^2 - | \nabla | \bar{A} ||^2$. We
  compute
\begin{align}
& | \bar{A} |^2 \mathcal{T} \\
= & | \bar{A} |^2 | \nabla \bar{A} |^2 - \tfrac{1}{4} | \nabla | \bar{A}
|^2 |^2 \\
= & \sum_{i, j, k, l, m} (\bar{h}_{i j} \nabla_k \bar{h}_{m l}) - \sum_k
(\sum_{i j} \bar{h}_{i j} \nabla_k \bar{h}_{i j})^2 \\
= & \tfrac{1}{2} \sum_{i, j, k, l, m} (\bar{h}_{i j} \nabla_k \bar{h}_{m
l} - \bar{h}_{m l} \nabla_k \bar{h}_{i j})^2 .
\end{align}
  By choosing basis of $T M$ we can assume that $\bar{h}_{i j}$ is diagonal
  and
  \begin{equation}
    \bar{h}_{i j} = \lambda_i \delta_{i j} . \label{diagonal of h bar}
  \end{equation}
  Using \eqref{diagonal of h bar}, we have
\begin{align}
& \sum_{i, j, k, l, m} (\bar{h}_{i j} \nabla_k \bar{h}_{m l} - \bar{h}_{m
l} \nabla_k \bar{h}_{i j})^2 \\
= & \sum_{i, m, l, k} (\bar{h}_{i i} \nabla_k \bar{h}_{m l} - \bar{h}_{m
l} \nabla_k \bar{h}_{i i})^2 + (\sum_{m, l} \bar{h}_{m l}^2) \sum_{i \neq
j, k} | \nabla_k \bar{h}_{i j} |^2 \\
\geqslant & (\sum_i \bar{h}_{i i}^2) \sum_{m \neq l, k} | \nabla_k
\bar{h}_{m l} |^2 + (\sum_{m, l} \bar{h}_{m l}^2) \sum_{i \neq j, k} |
\nabla_k \bar{h}_{i j} |^2 \\
= & 2 \left( \sum_{m, l} \bar{h}_{m l}^2 \right) \sum_{i \neq j, k} |
\nabla_k \bar{h}_{i j} |^2 \\
= & 2 | \bar{A} |^2 \sum_{i \neq j, k} | \nabla_k \bar{h}_{i j} |^2 .
\end{align}
  Hence,
  \begin{equation}
    \mathcal{T}= | \nabla \bar{A} |^2 - | \nabla | \bar{A} ||^2 \geqslant
    \sum_{i \neq j, k} | \nabla_k \bar{h}_{i j} |^2 .
  \end{equation}
  Observe that the proof of the above inequality only uses symmetry properties
  of $\bar{h}_{i j}$. We estimate $\sum_{i \neq j, k} | \nabla_k \bar{h}_{i j}
  |^2$ as follows:
  \begin{equation}
    \sum_{i \neq j, k} | \nabla_k \bar{h}_{i j} |^2 \geqslant \sum_{i \neq j}
    | \nabla_i \bar{h}_{i j} |^2 + | \nabla_j \bar{h}_{i j} |^2 = 2 \sum_{i
    \neq j} | \nabla_i \bar{h}_{i j} |^2 .
  \end{equation}
  Since
  \begin{equation}
    \sqrt{\sum_{i \neq j} | \nabla_j \bar{h}_{i i} |^2} \leqslant
    \sqrt{\sum_{i \neq j} | \nabla_i \bar{h}_{i j} - \nabla_j \bar{h}_{i i}
    |^2} + \sqrt{\sum_{i \neq j} | \nabla_i \bar{h}_{i j} |^2},
  \end{equation}
  and by Codazzi equation and the definition of $\bar{h}$ on the first term on
  the right hand side, we have
  \begin{equation}
    \sqrt{\sum_{i \neq j} | \nabla_i \bar{h}_{i j} |^2} \geqslant
    \sqrt{\sum_{i \neq j} | \nabla_j \bar{h}_{i i} |^2} - \sqrt{\sum_{i \neq
    j} (K_{n + 1, i j i} + \nabla_i \bar{T}_{i j} - \nabla_j \bar{T}_{i i})^2}
    .
  \end{equation}
  With the bound on $K$ by $|K|_{C^0}$ and \eqref{derivatives of T} in the
  above, we obtain
  \begin{equation}
    \sqrt{\sum_{i \neq j} | \nabla_i \bar{h}_{i j} |^2} \geqslant
    \sqrt{\sum_{i \neq j} | \nabla_j \bar{h}_{i i} |^2} - c (1 + |A|) .
  \end{equation}
  By invoking the elementary fact that the inequality $\sqrt{a} \geqslant
  \sqrt{b} - \sqrt{c}$ implies $a \geqslant \tfrac{b}{1 + \varepsilon} -
  \tfrac{c}{\varepsilon}$, we have
\begin{align}
& \sum_{i \neq j, k} | \nabla_k \bar{h}_{i j} |^2 \\
\geqslant & 2 \sum_{i \neq j} | \nabla_i \bar{h}_{i j} |^2 \\
\geqslant & \tfrac{2}{1 + \varepsilon} \sum_{i \neq j} | \nabla_j
\bar{h}_{i i} |^2 - \tfrac{c}{\varepsilon} (1 + |A|)^2 .
\end{align}
  We then try to bound $\sum_{i \neq j} | \nabla_j \bar{h}_{i i} |$ using $|
  \nabla | \bar{A} ||$. By \eqref{diagonal of h bar},
\begin{align}
& | \nabla | \bar{A} ||^2 \\
= & | \bar{A} |^{- 2} \sum_k (\sum_{j, i} \bar{h}_{i j} \nabla_k
\bar{h}_{i j})^2 \\
= & (\sum_{\ell} \bar{h}_{\ell \ell}^2)^{- 1} \sum_k (\sum_i \bar{h}_{i i}
\nabla_k \bar{h}_{i i})^2 \\
\leqslant & \sum_{i, k} | \nabla_k \bar{h}_{i i} |^2 \\
= & \sum_{i \neq k} | \nabla_k \bar{h}_{i i} |^2 + \sum_i | \nabla_i
\bar{h}_{i i} |^2 \\
= & \sum_{i \neq k} | \nabla_k \bar{h}_{i i} |^2 + \sum_i | \nabla_i h_{i
i} + \nabla_i \bar{T}_{i i} |^2 .
\end{align}
  Because of $H = \sum_i h_{i i}$, the bounds \eqref{first derivative of H}
  and \eqref{derivatives of T},
\begin{align}
& | \nabla | \bar{A} ||^2 \\
= & \sum_{i \neq k} | \nabla_k \bar{h}_{i i} |^2 + \sum_i | \sum_{j \neq
i} - \nabla_i h_{j j} + \nabla_i H + \nabla_i \bar{T}_{i i} |^2
\\
= & \sum_{i \neq k} | \nabla_k \bar{h}_{i i} |^2 + \sum_i | \sum_{j \neq
i} - \nabla_i \bar{h}_{j j} + \nabla_i \bar{T}_{j j} + \nabla_i H +
\nabla_i \bar{T}_{i i} |^2 \\
= & \sum_{i \neq k} | \nabla_k \bar{h}_{i i} |^2 + \sum_i (\sum_{j \neq i}
\nabla_i \bar{h}_{j j})^2 + c | \nabla \bar{A} | |A| + c |A|^2 + c |
\nabla \bar{A} | + c \\
\leqslant & n \sum_{i \neq j} | \nabla_i \bar{h}_{j j} |^2 + c | \nabla
\bar{A} | |A| + c |A|^2 + c | \nabla \bar{A} | + c.
\end{align}
  Therefore,
\begin{align}
& \sum_{i j k} | \nabla_k \bar{h}_{i j} |^2 - | \nabla | \bar{A} ||^2
\\
\geqslant & \sum_{i \neq j, k} | \nabla_k \bar{h}_{i j} |^2 \\
\geqslant & \tfrac{2}{1 + \varepsilon} \sum_{i \neq j} | \nabla_j
\bar{h}_{i i} |^2 - \tfrac{c}{\varepsilon} (1 + |A|)^2 \\
\geqslant & \tfrac{2}{(1 + \varepsilon) n} | \nabla | \bar{A} ||^2 -
\tfrac{c}{\varepsilon} (1 + |A|)^2 - (c| \nabla \bar{A} ||A| + c|A|^2 + c|
\nabla \bar{A} | + c) .
\end{align}
  By absorbing $|A|$ into $|A|^2$, $| \nabla \bar{A} |$ to the left, and
  \eqref{perturbed and original relation}, we obtained the desired inequality.
\end{proof}

\begin{corollary}
  Let $M$ be a MOTS, then
  \begin{equation}
    | \nabla \bar{A} |^2 - | \nabla | \bar{A} ||^2 \geqslant \frac{1}{(1 +
    \varepsilon) n + 1} (| \nabla \bar{A} |^2 + | \nabla | \bar{A} ||^2) - c
    (1 + | \bar{A} |^2), \label{kato variant}
  \end{equation}
  with $c$ depending on the same constants as in Lemma \ref{kato}.
\end{corollary}

\section{Curvature estimates}\label{curvature estimate section}

\subsection{$L^q$ curvature estimate}

\begin{proposition}
  \label{integral curvature estimate}Let $M$ be a stable free boundary MOTS,
  we have that for $q \in [2, 2 + \sqrt{\tfrac{8}{n}})$,
  \begin{equation}
    \int_M \phi^2 | \bar{A} |^{q - 2} | \nabla | \bar{A} ||^2 \leqslant c
    \int_M (\phi^2 + | \nabla \phi |^2) | \bar{A} |^q,
  \end{equation}
  and
  \begin{equation}
    \int \phi^2 | \bar{A} |^{q + 2} \leqslant c \int_M | \bar{A} |^q [| \nabla
    \phi |^2 + \phi^2], \label{cutoff lp estimate}
  \end{equation}
  where the constant $c > 0$ depends only on $|K|_{C^0}$, $|p|_{C^1}$,
  $|b|_{C^1}$, $| \eta |_{C^1}$ and $|d|_{C^2}$.
\end{proposition}

\begin{proof}
  We have respectively
  \[ \Delta | \bar{A} |^2 = 2 | \bar{A} | \Delta | \bar{A} | + 2 | \nabla |
     \bar{A} ||^2, \]
  and
  \[ \Delta | \bar{A} |^2 = 2 \bar{h}_{i j} \Delta \bar{h}_{i j} + 2 | \nabla
     \bar{A} |^2 . \]
  Subtracting the above two equations give
  \begin{equation}
    - | \bar{A} | \Delta | \bar{A} | + | \nabla \bar{A} |^2 - | \nabla |
    \bar{A} ||^2 = - \bar{h}_{i j} \Delta \bar{h}_{i j} .
  \end{equation}
  Recall the Simons inequality \eqref{simons for bar h}, we have that
\begin{align}
& - | \bar{A} | \Delta | \bar{A} | + | \nabla \bar{A} |^2 - | \nabla |
\bar{A} ||^2 \\
\leqslant & (\nabla K + \nabla^2 H + \Delta \bar{T}) \ast \bar{h} + c (1 +
| \bar{A} |^3 + | \nabla \bar{A} || \bar{A} |) + | \bar{A} |^4 .
\end{align}
  Multiply this equation by $\phi^2 | \bar{A} |^q$ and integrate. This yield
\begin{align}
& \int_M - \phi^2 | \bar{A} |^{q - 1} \Delta | \bar{A} | + \phi^2 |
\bar{A} |^{q - 2} (| \nabla \bar{A} |^2 - | \nabla | \bar{A} ||^2)
\\
\leqslant & \int_M \phi^2 | \bar{A} |^{q + 2} + c \phi^2 (1 + | \bar{A}
|^3 + | \nabla \bar{A} || \bar{A} |) |A|^{q - 2} \\
& + \int_M \phi^2 | \bar{A} |^{q - 2} (\nabla K + \nabla^2 H + \Delta
\bar{T}) \ast \bar{h} .
\end{align}
  Doing an integration by parts involving the Laplacian on the first line and
  $\nabla K + \nabla^2 H + \Delta \bar{T}$ on the last line, performing some
  elementary estimates we find
\begin{align}
& \int_M (q - 1) \phi^2 | \bar{A} |^{q - 2} | \nabla | \bar{A} ||^2 +
\phi^2 | \bar{A} |^{q - 2} (| \nabla \bar{A} |^2 - | \nabla | \bar{A}
||^2) \\
\leqslant & \int_M \phi^2 | \bar{A} |^{q + 2} + c \phi^2 (1 + | \bar{A}
|^3 + | \nabla \bar{A} || \bar{A} |) |A|^{q - 2} \\
& - \int_M 2 \phi | \bar{A} |^{q - 2} \langle \nabla | \bar{A} |, \nabla
\phi \rangle + \int_{\partial M} \phi^2 | \bar{A} |^{q - 1}
\partial_{\eta} | \bar{A} | \\
& + c \int_M \phi | \nabla \phi | U | \bar{A} |^{q - 1} + \phi^2 U |
\bar{A} |^{q - 2} (| \nabla | \bar{A} || + | \nabla \bar{A} |) \\
& + \int_{\partial M} \phi^2 | \bar{A} |^{q - 1} U.
\end{align}
  Here
  \begin{equation}
    U : = |K| + | \nabla H| + | \nabla \bar{T} | \leqslant c | \bar{A} |
    \label{U}
  \end{equation}
  due to \eqref{derivatives of T} and \eqref{first derivative of H}. After
  applying \eqref{boundary derivative bound} to the above, then
  \eqref{transfer of boundary integral}, we obtain,
\begin{align}
& \int_M (q - 1) \phi^2 | \bar{A} |^{q - 2} | \nabla | \bar{A} ||^2 +
\phi^2 | \bar{A} |^{q - 2} (| \nabla \bar{A} |^2 - | \nabla | \bar{A}
||^2) \\
\leqslant & \int_M \phi^2 | \bar{A} |^{q + 2} + c \phi^2 (1 + | \bar{A}
|^3 + | \nabla \bar{A} || \bar{A} |) |A|^{q - 2} \\
& - \int_M 2 \phi | \bar{A} |^{q - 2} \langle \nabla | \bar{A} |, \nabla
\phi \rangle + c \int_M \phi^2 | \bar{A} |^q \\
& + c \int_M \phi | \nabla \phi |  | \bar{A} |^q + \phi^2 | \bar{A} |^{q
- 1} (| \nabla | \bar{A} || + | \nabla \bar{A} |)
\end{align}
  We have the following inequalities for any $s < 2$ and any small
  $\varepsilon > 0$ which follow from an application of the Young's
  inequality, and we apply them \ to the above,
\begin{align}
| \bar{A} |^{q + s} & \leqslant \varepsilon | \bar{A} |^{q + 2} + c
(\varepsilon^{- 1}, s), \\
| \nabla \bar{A} |  | \bar{A} |^{q - 1} & \leqslant \varepsilon | \nabla
\bar{A} |^2 | \bar{A} |^{q - 2} + c (\varepsilon^{- 1}) | \bar{A} |^q,
\\
\phi | \langle \nabla | \bar{A} |, \nabla \phi \rangle | & \leqslant
\varepsilon \phi^2 | \nabla | \bar{A} ||^2 + c (\varepsilon^{- 1}) |
\nabla \phi |^2, \\
\phi | \nabla \phi |  | \bar{A} |^q & \leqslant \varepsilon \phi^2 |
\bar{A} |^{q + 2} + c (\varepsilon^{- 1}) | \nabla \phi |^2 | \bar{A} |^q,
\\
| \nabla | \bar{A} || + | \nabla \bar{A} | & \leqslant \varepsilon (|
\nabla | \bar{A} ||^2 + | \nabla \bar{A} |^2) + c (\varepsilon^{- 1})
\\
& \leqslant \varepsilon (| \nabla | \bar{A} ||^2 + | \nabla \bar{A} |^2)
+ c (\varepsilon^{- 1}) .
\end{align}
  Hence we obtain
\begin{align}
& \int_M (q - 1) \phi^2 | \bar{A} |^{q - 2} | \nabla | \bar{A} ||^2 +
\phi^2 | \bar{A} |^{q - 2} (| \nabla \bar{A} |^2 - | \nabla | \bar{A}
||^2) \\
\leqslant & (1 + \varepsilon) \int_M \phi^2 | \bar{A} |^{q + 2} +
\varepsilon \int_M \phi^2 | \bar{A} |^{q - 2} (| \nabla | \bar{A} ||^2 + |
\nabla \bar{A} |^2) \\
& + c \int_M (| \nabla \phi |^2 + \phi^2) | \bar{A} |^q . \label{estimate
with nabla A bar}
\end{align}
  We use the inequalities \eqref{eq:lp estimate by gradient}, \eqref{kato
  inequality} and \eqref{kato variant}, we obtain
\begin{align}
& \int_M (q - 1 + \frac{2}{(1 + \varepsilon) n} - \varepsilon \frac{1}{(1
+ \varepsilon) n + 1}) \phi^2 | \bar{A} |^{q - 2} | \nabla | \bar{A} ||^2
\\
\leqslant & (1 + \varepsilon) (1 + \varepsilon) \tfrac{q^2}{4} \int_M
\phi^2 | \bar{A} |^{q - 2} | \nabla | \bar{A} ||^2 + c \int_M (| \nabla
\phi |^2 + \phi^2) | \bar{A} |^q .
\end{align}
  Since $q \in [2, 2 + \sqrt{\tfrac{8}{n}})$ ensures that $\tfrac{q^2}{4} < q
  - 1 + \tfrac{2}{n}$, by choosing $\varepsilon$ sufficiently small, we can
  absorb the first term on the right to the left, and we obtain the desired
  inequality.
  
  The second inequality which asserts a bound on $\int \phi^2 | \bar{A} |^{q +
  2}$ follows by combining with Lemma \ref{lp estimate by gradient}.
\end{proof}

\begin{theorem}
  Let $q \in [2, 2 + \sqrt{\tfrac{8}{n}})$, and $M$ be a stable free boundary
  MOTS satisfying the volume bound \eqref{volume growth}, then
  \begin{equation}
    \int_{B (x, r / 4)} | \bar{A} |^{q + 2} \leqslant c r^{n - 2 - q},
    \label{lp estimate in ball}
  \end{equation}
  where the constant $c$ depends on $|\ensuremath{\operatorname{Ric}}|_{C^0}$,
  $|p|_{C^1}$, $|b|_{C^0}$, $|d|_{C^2}$ and the constant $C_M$ in
  \eqref{volume growth}.
\end{theorem}

\begin{proof}
  By letting $q = 2$ and $\phi$ be a standard cutoff, that is, $\phi = 0$
  outside $B (x, r)$, $\phi = 1$ in $B (x, r / 2)$, $| \nabla \phi | \leqslant
  \tfrac{c}{r}$. So using $\phi$ in \eqref{initial l2 estimate no boundary},
  \begin{equation}
    \int_{B (x, r / 2)} | \bar{A} |^4 \leqslant c r^{- 2} \int_{B (x, r)} |
    \bar{A} |^2 \leqslant c r^{- 4 + n} . \label{l4 estimate}
  \end{equation}
  From the $L^q$ estimate \eqref{cutoff lp estimate}, we have
  \begin{equation}
    \int_{B (x, r / 4)} | \bar{A} |^{q + 2} \leqslant \tfrac{c}{r^2} \int_{B
    (x, r / 2)} | \bar{A} |^q .
  \end{equation}
  Note that $q < 2 + \sqrt{\tfrac{8}{n}} \leqslant 4$, so from H{\"o}lder
  inequality and the $L^4$ estimate \eqref{l4 estimate},
\begin{align}
& \int_{B (x, r / 4)} | \bar{A} |^{q + 2} \\
\leqslant & \tfrac{c}{r^2} \int_{B (x, r / 2)} | \bar{A} |^q . \\
\leqslant & \tfrac{c}{r^2} \left( \int_{B (x, r / 2)} | \bar{A} |^4
\right)^{\tfrac{q}{4}} |B (x, r / 2) |^{1 - \tfrac{q}{4}} \\
\leqslant & c r^{- 2 - q + n} .
\end{align}
  This is our desired bound.
\end{proof}

Now we prove our main Theorem \ref{main curvature estimate}.

\begin{proof}[Proof of Theorem \ref{main curvature estimate}]
  We employ the iteration method of De Giorgi. Recall the Simons inequality
  \eqref{simons for bar h}, we have that
\begin{align}
& - | \bar{A} | \Delta | \bar{A} | + | \nabla \bar{A} |^2 - | \nabla |
\bar{A} ||^2 \\
\leqslant & (\nabla K + \nabla^2 H + \Delta \bar{T}) \ast \bar{h} + c (1 +
| \bar{A} |^3 + | \nabla \bar{A} || \bar{A} |) + | \bar{A} |^4 .
\end{align}
  Applying \eqref{kato} and Cauchy-Schwarz inequality to absorb the term $|
  \nabla \bar{A} |  | \bar{A} |$, and absorbing $| \bar{A} |^3$ into $|
  \bar{A} |^4$,  we have that
  \begin{equation}
    - \Delta | \bar{A} |^2 + c_1 | \nabla \bar{A} |^2 \leqslant (\nabla K +
    \nabla^2 H + \Delta \bar{T}) \ast \bar{h} + c | \bar{A} |^4 .
  \end{equation}
  Here $c_1$ is a positive constant. We multiply both sides by $\phi$, we have
  that
  \begin{equation}
    - \int_M \phi \Delta | \bar{A} |^2 + c_1 \phi  | \nabla \bar{A} |^2
    \leqslant \int_M \phi (\nabla K + \nabla^2 H + \Delta \bar{T}) \ast
    \bar{h} + c \phi | \bar{A} |^4 .
  \end{equation}
  We use integration by parts on the first term on the right and the bound
  \eqref{U} on $U$ to obtain a bound on the first term on the right:
\begin{align}
& \int_M \phi (\nabla K + \nabla^2 H + \Delta \bar{T}) \ast \bar{h}
\\
= & \int_{\partial M} \phi (K + \nabla H + \nabla \bar{T}) \ast \bar{h}
\ast \eta - \int_M (K + \nabla H + \nabla \bar{T}) \ast (\bar{h} \nabla
\phi + \phi \nabla \bar{h}) \\
\leqslant & c \int_{\partial M} \phi | \bar{A} |^2 + c \int_M | \nabla
\phi |  | \bar{A} |^2 + c \int_M \phi | \bar{A} |  | \nabla \bar{A} | .
\end{align}
  Absorbing $| \nabla \bar{A} |$ using the Cauchy-Schwarz inequality
  \[ | \bar{A} |  | \nabla \bar{A} | \leqslant \varepsilon | \nabla \bar{A}
     |^2 + \tfrac{1}{2 \varepsilon} | \bar{A} |^2, \]
  absorbing $| \bar{A} |^2$ into $| \bar{A} |^4$, we have
  \begin{equation}
    - \int_M \phi \Delta | \bar{A} |^2 \leqslant c \int_{\partial M} \phi |
    \bar{A} |^2 + c \int_M | \nabla \phi |  | \bar{A} |^2 + \phi | \bar{A} |^4
    .
  \end{equation}
  Let $u = | \bar{A} |^2$, $v = \max \{u - k, 0\}$ and replacing $\phi$ by
  $\phi^2 v$ in the above,
  \begin{equation}
    - \int_M \phi^2 v \Delta u \leqslant c \int_{\partial M} \phi^2 v u + c
    \int_M | \nabla (\phi^2 v) | u + \phi^2 v u | \bar{A} |^2 .
  \end{equation}
  We can apply divergence theorem on the first term
  \begin{equation}
    - \int_M \phi^2 v \Delta u = - \int_{\partial M} \phi^2 v \partial_{\eta}
    v + \int_M \phi^2 | \nabla v|^2 + 2 \phi v \langle \nabla \phi, \nabla v
    \rangle .
  \end{equation}
  Note that $| \partial_{\eta} v| \leqslant c u$ due to \eqref{boundary
  derivative bound}, so
\begin{align}
& \int_M \phi^2 | \nabla v|^2 + 2 \phi v \langle \nabla \phi, \nabla v
\rangle + \int_M c_1 \phi^2 v  | \nabla \bar{A} |^2 \\
\leqslant & c \int_{\partial M} \phi^2 v u + c \int_M | \nabla (\phi^2 v)
| u + \phi^2 v u | \bar{A} |^2 .
\end{align}
  For the boundary term $\int_{\partial M} \phi^2 v u$, we use
\begin{align}
& \int_{\partial M} \phi^2 v u \\
= & \int_M \ensuremath{\operatorname{div}}_M (\phi^2 v u \nabla d)
\\
= & \int_M 2 \phi \langle \nabla \phi, \nabla d \rangle v u + \int_M
\phi^2 u \langle \nabla v, \nabla d \rangle + \int_M \phi^2 v \langle
\nabla u, \nabla d \rangle + \phi^2 u v\ensuremath{\operatorname{div}}_M
\nabla d \\
\leqslant & c \int_M \phi | \nabla \phi | v u + \int_M \phi^2 u | \nabla
v| + \int_M \phi^2 v | \nabla u| + \int_M \phi^2 u v \\
\leqslant & c \int_M v^2 | \nabla \phi |^2 + c \int_M \phi^2 u^2 +
\varepsilon \int_M \phi^2 | \nabla v|^2 + c (\varepsilon^{- 1}) \int_M
\phi^2 u^2 \\
& + \varepsilon \int_M \phi^2 | \nabla v|^2 + c (\varepsilon^{- 1})
\int_M \phi^2 u^2 + \int_M \phi^2 u^2 \\
\leqslant & \varepsilon \int_M \phi^2 | \nabla v|^2 + c \int_M v^2 |
\nabla \phi |^2 + c \int_M k^2 \phi^2 + \int_M \phi^2 v^2 .
\end{align}
  And the term $c \int_M | \nabla (\phi^2 v) | u$ is estimated as follows:
\begin{align}
& \int_M | \nabla (\phi^2 v) | u \\
\leqslant & 2 \int_M \phi | \nabla \phi | v u + \int_M \phi^2 | \nabla v|
u \\
\leqslant & c \int_M | \nabla \phi |^2 v^2 + \int_M \phi^2 u^2 +
\varepsilon \int_M \phi^2 | \nabla v|^2 + c (\varepsilon^{- 1}) \int_M
\phi^2 u^2 \\
\leqslant & \varepsilon \int_M \phi^2 | \nabla v|^2 + c \int_M v^2 |
\nabla \phi |^2 + c \int_M k^2 \phi^2 + \int_M \phi^2 v^2 .
\end{align}
  And
  \[ \int_M \phi^2 v u | \bar{A} |^2 \leqslant \int_M \phi^2 v^2 | \bar{A} |^2
     + k \int_M \phi^2 v | \bar{A} |^2 \leqslant \tfrac{3}{2} \int_M \phi^2
     v^2 | \bar{A} |^2 + \tfrac{k^2}{2} \int_M \phi^2 | \bar{A} |^2 . \]
  Using a bound on $| \bar{A} |  \geqslant 1$,
\begin{align}
& \int_M \phi^2 | \nabla v|^2 + 2 \phi v \langle \nabla \phi, \nabla v
\rangle \\
\leqslant & 2 \varepsilon \int_M \phi^2 | \nabla v|^2 + c \int_M v^2 |
\nabla \phi |^2 + c \int_M k^2 | \bar{A} |^2 \phi^2 + \int_M \phi^2 v^2 .
\end{align}
  Letting $\varepsilon = \tfrac{1}{4}$, then we have that
\begin{align}
& \tfrac{1}{2} \int_M \phi^2 | \nabla v|^2 + 2 \phi v \langle \nabla
\phi, \nabla v \rangle \\
\leqslant & c \int_M v^2 | \nabla \phi |^2 + c \int_M k^2 | \bar{A} |^2
\phi^2 + c \int_M \phi^2 v^2 .
\end{align}
  Since
  \begin{equation}
    | \nabla (\phi v) |^2 = \phi^2 | \nabla v|^2 + 2 \phi v \langle \nabla v
    {,} \nabla \phi \rangle + v^2 | \nabla \phi |^2,
  \end{equation}
  and the Cauchy-Schwarz inequality $2 \phi v \langle \nabla \phi, \nabla v
  \rangle \geqslant - \tfrac{1}{4} \phi^2 | \nabla v|^2 + 4 v^2 | \nabla \phi
  |^2$, so
  \begin{equation}
    \int_M | \nabla (\phi v) |^2 \leqslant c \int_M v^2 | \nabla \phi |^2 + c
    \int_M k^2 | \bar{A} |^2 \phi^2 + c \int_M \phi^2 v^2 .
  \end{equation}
  Now the inequality is a good starting point for De Giorgi's iteration
  scheme. One more ingredient we need is the Sobolev inequality (Theorem
  \ref{sobolev}). \ For the details, we refer to the {\cite[Chapter
  4]{han-elliptic-2011}}. We obtain therefore an $L^2$ mean value inequality
  for $u$,
  \begin{equation}
    \sup_{B (x, r / 2)} u \leqslant c r^{- 2} \int_{B (x, r)} u^2
  \end{equation}
  provided that $| \bar{A} |^2 \in L^{q_1} (M)$ where $q_1 > n / 2$. By the
  estimate \eqref{lp estimate in ball}, $|A' | \in L^{q_2}$ for any $4
  \leqslant q_2 < 4 + \sqrt{\tfrac{8}{n}}$. Such $q_1$ exists only when $4 +
  \sqrt{\tfrac{8}{n}} > n$, that is, $n$ is in any dimension from 2 to 5. So
  when $2 \leqslant n \leqslant 5$,
  \begin{equation}
    \sup_{B (x, r)} | \bar{A} | \leqslant \tfrac{c}{r},
  \end{equation}
  finishing our proof the curvature estimate. Considering \eqref{perturbed and
  original relation}, we have the bound \eqref{curvature estimate} for $|A|$.
\end{proof}

\appendix\section{Sobolev inequalities\label{app:Sobolev}}

Recall the following general Sobolev inequality for hypersurfaces.

\begin{theorem}[{\cite{hoffman-sobolev-1974}}]
  \label{sobolev}Assume that $M$ is $n$-dimensional hypersurface in a manifold
  $\tilde{N}$, let $h$ be a nonnegative $C^1$ function on $M$ which vanishes
  on $\partial N$. Then
  \begin{equation}
    \|h\|_{L^{\tfrac{n}{n - 1}} (M)} \leqslant c \int_M | \nabla h| + h |H|,
  \end{equation}
  provided the measure of the support of $h$ is less than a constant $c_0 > 0$
  which depends $n$, the upper bound of the sectional curvature and the
  injective radius of $\tilde{N}$. Here $c$ depends only on $n$.
\end{theorem}

Going through the same reasoning as in {\cite[Theorem
2.3]{edelen-convexity-2016}}, we obtain the following.

\begin{theorem}
  \label{Sobolev}If $M$ meets $\partial N$ orthogonally and $v \in C^1
  (\bar{M})$ and the measure of the support of $v$ is less than $c_0$ (as in
  as Theorem \ref{sobolev}), then for any $1 \leqslant p < n$,
  \begin{equation}
    \|v\|_{L^{p^{\ast}} (M)} \leqslant c (\| \nabla v\|_{L^p (M)} +\|H
    v\|_{L^p (M)} +\|v\|_{L^p (M)}),
  \end{equation}
  where $c$ is a constant depending only the dimension $n$, the exponent $p$,
  the distance function to $\partial N$, the upper bound of the sectional
  curvature and the injective radius of $\tilde{N}$. The number $p^{\ast}$ is
  the critical Sobolev exponent, when $n > 2$, $p^{\ast} = \tfrac{n p}{n - p}$
  and when $n = 2$, $p^{\ast}$ could be any number greater than 2.
\end{theorem}

\end{document}